\documentclass[10pt,a4paper,final]{amsart}
\usepackage[cp1250]{inputenc}
\usepackage[OT4]{fontenc}
\usepackage{amsfonts}
\usepackage{hyperref}
\usepackage{nomencl}
\usepackage{amssymb}
\usepackage{graphicx}
\usepackage{color}
\usepackage{enumerate}

\newtheorem{tw}{Theorem}[section]
\newtheorem*{twm}{Main Theorem}
\newtheorem{lem}[tw]{Lemma}

\newtheorem{prop}[tw]{Proposition}
\theoremstyle{definition}
\newtheorem{df}[tw]{Definition}
\theoremstyle{remark}
\newtheorem{rem}[tw]{Remark}
\newtheorem{ex}[tw]{Example}

\def\C{\mathbb{C}}
\def\c{\mathbb{C}}
\def\X{\mathbb{X}}
\def\x{\mathbf{X}}
\def\w{\widehat}
\def\a{\alpha}
\def\t{\widetilde}
\def\W{\mathbf{W}}
\def\T{\mathbf{T}}
\def\O{\mathcal{O}}
\def\o{\mathcal{O}_{\mathcal{S}}}
\def\P{\mathbb{P}}
\def\h{\mathbf{h}}
\def\Z{\mathbf{Z}}
\def\y{\mathbf{Y}}
\def\Y{\mathbb{Y}}
\def\I{\mathcal{T}^N_k}
\def\J{\mathcal{Y}^N_k}
\def\Az{A^{\alpha}_0}
\def\Aj{A^{\alpha}_1}
\def\Dz{D^{\alpha}_0}
\def\Dj{D^{\alpha}_1}
\def\D{\mathbb{D}}
\def\Arg{\textnormal{Arg}}
\def\merge{$\leftrightsquigarrow$}

\begin{document}
\title[A Hartogs type theorem for generalized
($\textit{N,k}$)-crosses]{A Hartogs type extension
theorem for generalized ($\textit{\textbf{N,k}}$)-crosses with pluripolar singularities}
\author[Ma{\l}gorzata Zaj\k{e}cka]{Ma{\l}gorzata Zaj\k{e}cka}
\address{\textnormal{Jagiellonian University\newline
\indent Faculty of Mathematics and Computer Science\newline
\indent Institute of Mathematics\newline
\indent {\L}ojasiewicza 6\newline
\indent 30-348 Krak\'ow}}

\address{\textnormal{Cracow University of Technology\newline
\indent Faculty of Physics, Mathematics and Computer Science\newline
\indent Institute of Mathematics\newline
\indent Warszawska 24\newline
\indent 31-155 Krak\'ow\newline
$ $}}

\email{malgorzata.zajecka@gmail.com}
\begin{abstract} The aim of this paper is to present an
extension theorem for $(N,k)$-crosses with pluripolar singularities.
\end{abstract}
\maketitle

\section{Introduction. Statement of the main result}

\subsection{Introduction}
\indent The topic of separately holomorphic functions have a long history in
complex analysis. The problem was first investigated by W. F. Osgood in
\cite{Osg 1899}. Seven years later F. Hartogs in \cite{Har 1906} proved his
famous theorem stating that every separately holomorphic function is, in fact,
holomorphic. Since then the interest switched to more general problem - whether
a function $f$ defined on a product of two domains $D\times G$, being separately
holomorphic on some subsets $A\subset D$ and $B\subset G$ is holomorphic on the whole $D\times G$
(see for example papers of M. Hukuhara \cite{Huk 1942} and T. Terada \cite{Ter
1967}) - which led to the question of possible holomorphic extension of a
function separately holomorphic on the objects
called crosses.

\indent In a recent paper \cite{Lew} A.
Lewandowski introduces an object called generalized $(N,k)$-cross $\T_{N,k}$,
being a generalization of $(N,k)$-cross defined by M. Jarnicki and P. Pflug in
\cite{JarPfl 2010}, and proves an extension theorem for this new type of cross
with analytic singularities. In this paper we will prove a similar extension
theorem for $\T_{N,k}$ crosses with pluripolar singularities, being a
generalization of Theorem 10.2.9 from \cite{JarPfl 2011} and Main Theorem from
\cite{JarPfl 2003}. We will also introduce other type of generalized
$(N,k)$-crosses called $\y_{N,k}$ crosses, being more natural object to consider
in light of Theorem \ref{main}. This theorem will
turn out to be a very strong tool, allowing us to prove two Hartogs-type
extension theorems for the functions separately holomorphic on $\x_{N,k}$,
$\T_{N,k}$ and $\y_{N,k}$ crosses, including the Main Theorem of this paper.

\indent The paper is divided into four sections. In the first section we define
generalized $(N,k)$-crosses and we state the Main Theorem. Section 2
contains some useful definitions and facts. Section 3 is dedicated to
$(N,k)$-crosses - their properties and recent cross theorems. It also contains
the statement of Theorem \ref{main} and the proof of
Main Theorem. In the last section we present the detailed proof of Theorem
\ref{main}.

\subsection{Generalized (\textit{N,k})-crosses and the main result}
Let $D_j$ be a Riemann domain over $\c^{n_j}$ and let $A_j\subset D_j$ be
locally pluriregular (see Definition \ref{relex}), $j=1,\ldots ,N$, where $N\geq
2$. For $\alpha=(\alpha_1,\ldots,\alpha_N)\in\{0,1\}^N$ and $B_j\subset
D_j$, $j=1,\ldots,N$, define: $$
\mathcal{X}_{\alpha}:=\mathcal{X}_{1,\alpha_1}\times\ldots\times\mathcal{X}_{N,\alpha_N},\quad
\mathcal{X}_{j,\alpha_j}:=\left\{
\begin{matrix}
D_j&\text{ when }&\alpha_j=1\\
A_j&\text{ when }&\alpha_j=0\\
\end{matrix}\right. ,\quad j=1,\ldots,N,\\
$$
$$
B_0^\a:=\prod_{j\in\{1,\ldots,N\}:\,\a_j=0}B_j,\quad
B_1^\a:=\prod_{j\in\{1,\ldots,N\}:\,\a_j=1}B_j.
$$

For $\a\in\{0,1\}^N$ we merge $c_0\in\Dz$ and $c_1\in\Dj$ into
$(\overset{\resizebox{2,3em}{1mm}{\merge}}{c_0,c_1})\in\prod\limits_{j=1}^n
D_j$ by putting variables in right places.

We also use the following convention: for $D\subset\Dz$, $G\subset\Dj$,
$\a\in\{0,1\}^N$, define
$$
\overset{\resizebox{3em}{1mm}{\merge}}{D\times
G}:=\{(\overset{\resizebox{1,6em}{1mm}{\merge}}{a,b}):\
a\in D,\ b\in G\}. $$

To simplify the notation let us define families
$$\I:=\{\a\in\{0,1\}^N:\ |\a|=k\},\quad\J:=\{\a\in\{0,1\}^N:\ 1\leq |\a|\leq
k\}.$$

\begin{df}
For a $k\in\{1,\ldots,N\}$ we define an \emph{$(N,k)$-cross}
$$
\x_{N,k}=\mathbb{X}_{N,k}((A_j,D_j)_{j=1}^N):=\bigcup_{\a\in\I}\mathcal{X}_{\a}.
$$
\end{df}

For $\a\in\J$ let
$\Sigma_\a\subset\Az$ and put
$$\mathcal{X}_{\a}^\Sigma:=\{z\in\mathcal{X}_{\alpha}:\
z_{\a}\not\in\Sigma_{\a}\},\quad\a\in\J,$$
where $z_{\a}$
denotes the projection of $z$ on $\Dz$. 

\begin{df}
We define a
\emph{generalized $(N,k)$-cross} $\T_{N,k}$
$$\T_{N,k}=\mathbb{T}_{N,k}((A_j,D_j)_{j=1}^N,(\Sigma_\a)_{\a\in\I}):=\bigcup_{\a\in\I}\mathcal{X}_{\a}^\Sigma$$
and a \emph{generalized $(N,k)$-cross} $\y_{N,k}$
$$\y_{N,k}=\Y_{N,k}((A_j,D_j)_{j=1}^N,(\Sigma_\a)_{\a\in\J}):=\bigcup_{\a\in\J}\mathcal{X}_{\a}^\Sigma.$$
\end{df}

Observe that always $\T_{N,k}\subset\y_{N,k}$.

\begin{ex}
To see the difference between generalized $\T_{N,k}$ and $\y_{N,k}$
consider for example $N=3$, $k=2$, and let
$$\Sigma_{(1,1,0)}=\{z_3\}\subset A_3,\quad\Sigma_{(1,0,1)}=\{z_2\}\subset
A_2,\quad\Sigma_{(0,1,1)}=\{z_1\}\subset A_1,$$
$$\Sigma_\a=\varnothing,\
\a\in\mathcal{Y}^3_2\setminus\mathcal{T}^3_2 .$$

\end{ex}
\bigskip

Observe that if for all $\a\in\J$
we have $\Sigma_\a=\varnothing$, then
$$\mathbb{T}_{N,k}((A_j,D_j)_{j=1}^N,(\Sigma_\a)_{\a\in\I})=\Y_{N,k}((A_j,D_j)_{j=1}^N,(\Sigma_\a)_{\a\in\J})$$
$$=\mathbb{X}_{N,k}((A_j,D_j)_{j=1}^N).$$
Moreover, for $k=1$ we have
$(\Sigma_\a)_{\a\in\I}=(\Sigma_\a)_{\a\in\J}=(\Sigma_j)_{j=1}^N$ and we use the
simplified notation
$$\T_{N,1}=\y_{N,1}=:\mathbb{T}((A_j,D_j,\Sigma_j)_{j=1}^N).$$

\begin{df}
For $(N,k)$-cross $\W_{N,k}\in\{\x_{N,k},\T_{N,k},\y_{N,k}\}$ we define its
\emph{center} as
$$
c(\W_{N,k}):=\W_{N,k}\cap(A_1\times\ldots\times A_N).
$$
\end{df}

\begin{df}
For a cross $\x_{N,k}=\mathbb{X}_{N,k}((A_j,
D_j)_{j=1}^N)$ we define its \emph{hull}
$$
\w{\x}_{N,k}=\widehat{\mathbb{X}}_{N,k}((A_j, D_j)_{j=1}^N):=\big\{(z_1,\ldots
,z_N)\in D_1\times\ldots\times D_N:\ \sum_{j=1}^N\h_{A_j,D_j}(z_j)<k\big\}, $$
where $\h_{B,D}$ denotes relative extremal function of $B$ with respect to $D$
(see Definition \ref{relex}).
\end{df}

Let $\W_{N,k}\in\{\T_{N,k},\y_{N,k}\}$ and let $M\subset\W_{N,k}$. For an 
$\a\in\J$ and for an $a\in\Az$ let
$M_{a,\a}$ denote the fiber $$M_{a,\alpha}:=\{z\in\Dj:\
(\overset{\resizebox{1,6em}{1mm}{\merge}}{a,z})\in M\}.$$

For $(z',z'')\in\prod\limits_{j=1}^kD_j\times\prod\limits_{j=k+1}^ND_j$,
$k\in\{1,\ldots,N-1\}$, define
$$
M_{(z',\cdot)}:=\{b\in\prod\limits_{j=k+1}^ND_j: (z',b)\in M\},\
M_{(\cdot,z'')}:=\{a\in\prod\limits_{j=1}^kD_j: (a,z'')\in M\}.$$

\begin{df}\label{holT}
Let $M\subset\T_{N,k}$ be such
that for all $\a\in\I$ and for all
$a\in\Az\setminus\Sigma_{\a}$ the set
$\Dj\setminus M_{a,\alpha}$ is open.
A function
$f:\T_{N,k}\setminus M\to\c$ is called \emph{separately holomorphic on}
$\T_{N,k}\setminus M$ ($f\in\o(\T_{N,k}\setminus M)$), if for all
$\a\in\I$ and for all
$a\in\Az\setminus\Sigma_{\a}$, the function
\begin{align*}
\tag{$\dag$}\Dj\setminus
M_{a,\alpha}\ni z\mapsto
f((\overset{\resizebox{1,6em}{1mm}{\merge}}{a,z}))=:f_{a,\a}(z)
\end{align*}
is holomorphic.
\end{df}

For generalized $(N,k)$-cross $\y_{N,k}$ we state an analogical definition.

\begin{df}\label{holY}
Let $M\subset\y_{N,k}$ be such
that for all $\a\in\J$ and for all
$a\in\Az\setminus\Sigma_{\a}$ the set
$\Dj\setminus M_{a,\alpha}$ is open.
A function
$f:\y_{N,k}\setminus M\to\c$ is called \emph{separately holomorphic on}
$\y_{N,k}\setminus M$ ($f\in\o(\y_{N,k}\setminus M)$), if for all
$\a\in\I$ and for all
$a\in\Az\setminus\Sigma_{\a}$, the function
($\dag$) is holomorphic.
\end{df}

\begin{rem}\label{hol_rem}
Observe that if $f\in\o(\y_{N,k}\setminus M)$, then also for all $\a\in\J$ and
for all $a\in\Az\setminus\Sigma_{\a}$ the function ($\dag$) is
holomorphic.
\end{rem}

Let $M\subset\T_{N,k}$. For
$\a\in\J$ and for $b\in\Dj$ let
$M_{b,\a}$ denote the fiber
$$M_{b,\alpha}:=\{z\in\Az:\
(\overset{\resizebox{1,6em}{1mm}{\merge}}{z,b})\in M\}.$$

The following class of functions plays important role in the Main Theorem. It is
a natural extension of class $\o(\T_{N,k}\setminus
M)\cap\mathcal{C}(\T_{N,k}\setminus M)$.

\begin{df}
Let $M\subset\T_{N,k}$ be such
that for all $\a\in\I$ and for all
$a\in\Az\setminus\Sigma_{\a}$ the set
$\Dj\setminus M_{a,\alpha}$ is open.
By $\o^c(\T_{N,k}\setminus M)$ we denote the space of all functions
$f\in\o(\T_{N,k}\setminus M)$ such that for all
$\a\in\I$ and for all
$b\in\Dj$, the function
$$\Az\setminus (\Sigma_\a\cup M_{b,\alpha})\ni
z\mapsto
f((\overset{\resizebox{1,6em}{1mm}{\merge}}{z,b}))=:f_{b,\a}(z)$$ is
continuous.
\end{df}

The following theorem is the main result of this paper. It is an
analogue and natural generalization of Theorem 10.2.9 from \cite{JarPfl 2011}. It also
extends the main result from \cite{Lew}.

\begin{twm}[Extension theorem for $(N,k)$-crosses with
pluripolar singularities]\label{result}
Let $D_j$ be a Riemann
domain of holomorphy over $\C^{n_j}$, $A_j\subset D_j$ be locally pluriregular,
$j=1,\ldots,N$. For $\alpha\in\I$ let
$\Sigma_\a\subset\Az$ be pluripolar.
Let
$$
\x_{N,k}:=\X_{N,k}((A_j,D_j)_{j=1}^N),\quad\T_{N,k}:=\mathbb{T}_{N,k}((A_j,D_j)_{j=1}^N,(\Sigma_{\a})_{\a\in\I}).
$$
Let $M$ be a relatively closed, pluripolar subset of
$\T_{N,k}$ such that for all $\a\in\I$ and all
$a\in\Az\setminus\Sigma_\a$ the fiber $M_{a,\a}$ is
pluripolar. Let $$\mathcal{F}:=\left\{
\begin{matrix}
\o(\x_{N,k}\setminus M),&\text{ if for any }\alpha\in\I\text{ we have }
\Sigma_\a=\varnothing\\
\o^c(\T_{N,k}\setminus M),&\text{ otherwise}
\end{matrix}\right. .$$
Then there exists a relatively closed pluripolar set $\w{M}\subset\w{\x}_{N,k}$
and a generalized $(N,k)$-cross
$\T'_{N,k}:=\mathbb{T}_{N,k}((A_j,D_j)_{j=1}^N,(\Sigma'_{\a})_{\a\in\I})\subset
\T_{N,k}$ with $\Sigma_\a\subset\Sigma'_\a\subset\Az$, $\Sigma'_\a$ pluripolar,
$\a\in\I$, such that:
\begin{itemize}
  \item $\w{M}\cap (c(\T_{N,k})\cup\T'_{N,k})\subset M$,
  \item for any $f\in\mathcal{F}$ there exists a function
  $\w{f}\in\O(\w{\x}_{N,k}\setminus\w{M})$ such that $\w{f}=f$ on
  $(c(\T_{N,k})\cup\T'_{N,k})\setminus M$,
  \item $\w{M}$ is singular with respect to $\{\w{f}:\
  f\in\mathcal{F}\}$\footnotemark\footnotetext{That is, for all $a\in \w{M}$
  and $U_a$-open neighborhood of $a$ there exists an $\w{f}\in\{\w{f}:\
  f\in\mathcal{F}\}$ such that $\w{f}$ does not extend holomorphically to
  $U_a$. For more details see \cite{JarPfl 2000}, Chapter 3.}\!,
  \item if $M=\varnothing$, then $\w{M}=\varnothing$,
  \item if for all $\a\in\I$ and all
$a\in\Az\setminus\Sigma_\a$ the fiber $M_{a,\a}$ is
thin in $\Dj$, then $\w{M}$ is analytic in $\w{\x}_{N,k}$.
\end{itemize}
\end{twm}

The following remark shows that Main Theorem can be stated analogously to
Theorem 10.2.9 from \cite{JarPfl 2011}.

\begin{rem}\label{resultrem}
Observe that for any relatively closed pluripolar set $M\subset\T_{N,k}$ and for
all $\a\in\I$ there exists a pluripolar set
$\Sigma^0_\a\subset\Az$ such that $\Sigma_\a\subset\Sigma^0_\a$
and for all $a\in\Az\setminus\Sigma^0_\a$ the fiber $M_{a,\a}$
is pluripolar. Then from Main Theorem we get the conclusion with
$(\Sigma_\a)_{\a\in\I}$ and $\T_{N,k}$ substituted with
$(\Sigma^0_\a)_{\a\in\I}$ and
$\T^0_{N,k}:=\mathbb{T}_{N,k}((A_j,D_j)_{j=1}^N,(\Sigma^0_\a)_{\a\in\I}$ .
\end{rem}

\section{Preliminaries}

\subsection{Relative extremal function}

\begin{df}[Relative extremal function]\label{relex}
Let $D$ be a Riemann domain over $\C^n$ and let $A\subset D$. The
\emph{relative extremal function of $A$ with respect to $D$} is a function
$$\mathbf{h}_{A,D}:=\sup\{u\in\mathcal{PSH}(D):\ u\leq 1,\ u|_A\leq 0\}.$$
For an open set $G\subset D$ we define $\mathbf{h}_{A,G}:=\mathbf{h}_{A\cap
G,G}$.

\indent A set $A\subset D$ is called
\emph{pluriregular at a point $a\in\overline{A}$} if $\h^*_{A,U}(a)=0$ for any 
open neighborhood $U$ of the point $a$, where $\h^*_{A,U}$ denotes the upper
semicontinuous regularization of $\h_{A,U}$.

\indent We call $A$ \emph{locally
pluriregular} if $A\neq\varnothing$ and $A$ is pluriregular at every point $a\in A$.
\end{df}

\subsection{{\it N}-fold crosses}

Let $D_j$ be a Riemann domain over $\c^{n_j}$ and let $A_j\subset
D_j$ be a nonempty set, $j=1,\ldots ,N$, where $N\geq 2$. For $k=1$, for
historical reasons, we call $\X_{N,1}((A_j,D_j)_{j=1}^N)$ an \emph{N-fold cross}
$\x$ and we use the following notation
$$ \x=\mathbb{X}(A_1,\ldots,A_N;D_1,\ldots,D_N)=\mathbb{X}((A_j, D_j)_{j=1}^N)=
\bigcup_{j=1}^N(A_j'\times D_j\times A_j'') , $$
\noindent where
\begin{align*}
&A_j':=A_1\times\ldots\times A_{j-1},\ j=2,\ldots ,N,\\
&A_j'':=A_{j+1}\times\ldots\times A_{N},\ j=1,\ldots ,N-1,\\
&A_1'\times D_1\times A_1'':=D_1\times A_1'',\quad A_N'\times D_N\times
A_N'':=A_N'\times D_N.
\end{align*}

For $\Sigma_j\subset A'_j\times A''_j$, $j=1,\ldots,N$ put
$$\mathcal{X}_j:=\{(a'_j,z_j,a''_j)\in A'_j\times D_j\times A''_j:\
(a'_j,a''_j)\not\in\Sigma_j\},$$where
\begin{align*}
&a_j':=(a_1,\ldots ,a_{j-1}),\ j=2,\ldots ,N,\\
&a_j'':=(a_{j+1},\ldots , a_{N}),\ j=1,\ldots ,N-1,\\
&(a_1',z_1,a_1''):=(z_1,a_1''),\quad (a_N',z_N,a_N''):=(a_N',z_N).
\end{align*}
We call
$\mathbb{T}_{N,1}((A_j,D_j,\Sigma_j)_{j=1}^N) =\bigcup_{j=1}^N\mathcal{X}_j$ a
\emph{generalized N-fold cross} $\T$.

For $(a_j',a_j'')\in A_j'\times A_j''$, $j=1,\ldots,N$, define the fiber
$$
M_{(a_j',\cdot,a_j'')}:=\{z\in D_j:\ (a_j',z,a_j'')\in M\}.$$

Our proof of Main Theorem will be based on more technically complicated result
(i.e. Theorem \ref{main}), being an analogue of the following theorem. Moreover,
we will use this result as the first inductive step in the proof of
mentioned Theorem \ref{main}.

\begin{tw}[see \cite{JarPfl 2007}, Theorem 1.1]\label{classic}
Let $D_j$ be a Riemann domain of holomorphy over $\C^{n_j}$, $A_j\subset
D_j$ be locally pluriregular and let $\Sigma_j\subset
A'_j\times A''_j$ be pluripolar, $j=1,\ldots ,N$. Put
$$\x:=\X((A_j,D_j)_{j=1}^N),\quad\T:=\mathbb{T}((A_j,D_j,\Sigma_j)_{j=1}^N).$$
Let $\mathcal{F}\subset\{f:\ f:c(\T)\setminus M\to \C\}$ and let $M\subset\T$ be
such that:
\begin{itemize}
  \item for any $j\in\{1,\ldots,N\}$ and any $(a'_j,a''_j)\in (A'_j\times
  A''_j)\setminus\Sigma_j$ the fiber $M_{(a'_j,\cdot,a''_j)}$ is pluripolar,
  \item for any $j\in\{1,\ldots,N\}$ and any $(a'_j,a''_j)\in (A'_j\times
  A''_j)\setminus\Sigma_j$ there exists a closed pluripolar
  set $\t{M}_{a,j}\subset D_j$ such that $\t{M}_{a,j}\cap
  A_j\subset M_{(a'_j,\cdot,a''_j)}$,
  \item for any $a\in c(\T)\setminus M$ there exists an $r>0$ such
  that for all $f\in\mathcal{F}$ there exists an $f_a\in\O(\P(a,r))$ with
  $f_a=f$ on $\P(a,r)\cap(c(\T)\setminus M)$\footnotemark\footnotetext{$\mathbb{P}(a,r)$
  denotes a polydisc in Riemann domain $D_1\times\ldots\times D_N$ centered at
  $a$ with radius $r$. For more details see \cite{JarPfl 2000}, Chapter 1.},
  \item for any $f\in\mathcal{F}$, any $j\in\{1,\ldots,N\}$, 
  and any $(a'_j,a''_j)\in (A'_j\times A''_j)\setminus\Sigma_j$ there
  exists a function $\t{f}_{a,j}\in\O(D_j\setminus\t{M}_{a,j})$ such that
  $\t{f}_{a,j}=f(a'_j,\cdot,a''_j)$ on $A_j\setminus
  M_{a,j}$.
\end{itemize}
Then there exists a relatively closed pluripolar set $\w{M}\subset\w{\x}$
such that:
\begin{itemize}
  \item $\w{M}\cap c(\T)\subset M$,
  \item for any $f\in\mathcal{F}$ there exists a function
  $\w{f}\in\O(\w{\x}\setminus\w{M})$ such that $\w{f}=f$ on $c(\T)\setminus M$,
  \item $\w{M}$ is singular with respect to $\{\w{f}:\ f\in\mathcal{F}\}$,
  \item if for all $j\in\{1,\ldots,N\}$ and all $(a'_j,a''_j)\in (A'_j\times
  A''_j)\setminus\Sigma_j$ we have $\t{M}_{a,j}=\varnothing$, then $\w{M}=\varnothing$,
  \item if for all $j\in\{1,\ldots,N\}$ and all $(a'_j,a''_j)\in (A'_j\times
  A''_j)\setminus\Sigma_j$ the set $\t{M}_{a,j}$ is thin in $D_j$, then $\w{M}$
  is analytic in $\w{\x}$.
\end{itemize}

\end{tw}

\section{({\it N,k})-crosses}

\subsection{Basic properties of ({\it N,k})-crosses}

The following properties will be implicitly used throughout the paper.

\begin{lem}[Properties of $(N,k)$-crosses, see \cite{JarPfl 2010}, Remark
5]$\quad$
\begin{enumerate}[\rm(i)]
  \item $\x_{N,1}=\X((A_j,D_j)_{j=1}^N)$,
  $\w{\x}_{N,1}=\w{\X}((A_j,D_j)_{j=1}^N)$,
  \item $\x_{N,k}$ is arcwise connected,
  \item $\w{\x}_{N,k}$ is connected,
  \item if $D_1,\ldots,D_N$ are Riemann domains of holomorphy, then
  $\w{\x}_{N,k}$ is a Riemann domain of holomorphy,
  \item $\x_{N,k}\subset\x_{N,k+1}$, $\w{\x}_{N,k}\subset\w{\x}_{N,k+1}$,
  $k=1,\ldots,N-1$,
  \item $\x_{N,k}=\X(\x_{N-1,k-1},A_N;\x_{N-1,k},D_N)$, $k=2,\ldots,N-1$, $N>2$.
\end{enumerate}
\end{lem}

The following technical lemmas will also be useful.

\begin{lem}[\cite{JarPfl 2010}, Lemma 4]\label{inc}
Let $D_j$ be a Riemann domain of holomorphy over $\c^{n_j}$ and $A_j\subset D_j$
be locally pluriregular, $j=1,\ldots ,N$.
Then for all $z=(z_1,\ldots ,z_N)\in\w{\x}_{N,k}$ we have:
$$
\h_{\w{\x}_{N,k-1},\w{\x}_{N,k}}(z)=\max\big\{0,\sum_{j=1}^N
\h_{A_j,D_j}(z_j)-k+1\big\}. $$
\end{lem}

\begin{lem}\label{equalh}
Let $D_j$ be a Riemann domain of holomorphy over $\c^{n_j}$ and $A_j\subset D_j$
be locally pluriregular, $j=1,\ldots ,N$. Then for $z\in\w{\x}_{N,k}$
$$
\h_{\x_{N,k-1},\w{\x}_{N,k}}(z)=\h_{\w{\x}_{N,k-1},\w{\x}_{N,k}}(z).$$
\end{lem}

\begin{proof}
The inequality "$\geq$" follows from properties of relative extremal function
(see \cite{JarPfl 2011}, Proposition 3.2.2). To show the opposite inequality fix
a $u\in\mathcal{PSH}(\w{\x}_{N,k})$ such that $u\leq 1$ and $u|_{\x_{N,k-1}}=0$.
Then $u|_{\w{\x}_{N,k-1}}\in\mathcal{PSH}(\w{\x}_{N,k-1})$ and
$u|_{\w{\x}_{N,k-1}}\leq\h_{\x_{N,k-1},\w{\x}_{N,k-1}}$. Using
analogous\footnote{Instead of classical cross theorem
for $N$-fold crosses we use cross theorem for $(N,k)$-crosses - see Theorem
\ref{cross}.} reasoning as in Proposition 5.1.8 (i) from \cite{JarPfl 2011} we
show that $\h_{\x_{N,k-1},\w{\x}_{N,k-1}}\equiv 0$ on $\w{\x}_{N,k-1}$, what finishes the proof.
\end{proof}

\subsection{Cross theorems for ({\it N,k})-crosses}

In this section we present the latest results considering
$(N,k)$-crosses which will be used in the proof of the Main Theorem. Observe
that our main result generalizes both of them.

\begin{tw}[Cross theorem for $(N,k)$-crosses, cf. \cite{JarPfl 2011}, Theorem
7.2.7]\label{cross}
Let $D_j$ be a Riemann domain of holomorphy over $\c^{n_j}$ and $A_j\subset D_j$
be locally pluriregular, $j=1,\ldots,N$. For $k\in\{1,\ldots,N\}$ let
$\x_{N,k}:=\mathbb{X}_{N,k}((A_j, D_j)_{j=1}^N)$. Then for every
$f\in\o(\x_{N,k})$ there exists a unique function $\w{f}\in\O(\w{\x}_{N,k})$
such that $\w{f}=f$ on $\x_{N,k}$.
\end{tw}

The following result is a special case of Theorem 2.12 from
\cite{Lew} being a cross theorem without singularities
for generalized $(N,k)$-crosses.

\begin{tw}[Cross theorem for generalized $(N,k)$-crosses]\label{arek}
Let $D_j$ be a Riemann domain over $\c^{n_j}$, $A_j\subset D_j$ be
pluriregular, $j=1,\ldots,N$. For $\alpha\in\I$ let $\Sigma_\alpha$ be a
subset of $\Az$. Then for every $f\in\o^c(\T_{N,k})$ there exists
an $\w{f}\in\O(\w{\x}_{N,k})$ such that $\w{f}=f$ on $\T_{N,k}$.
\end{tw}

\subsection{Extension theorem for generalized ({\it N,k})-crosses with
pluripolar singularities}

Now we state an already mentioned main technical result, being an analogue of
Theorem \ref{classic} which is crucial for the proof of the Main Theorem. Its
proof will be presented in Section \ref{proof}.

\begin{tw}\label{main}
Let $D_j$ be a Riemann domain of holomorphy over $\C^{n_j}$, $A_j\subset
D_j$ be locally pluriregular, $j=1,\ldots ,N$.
For $\alpha\in\J$ let $\Sigma_\alpha$ be a
pluripolar subset of $\Az$. Let $\W_{N,k}\in\{\x_{N,k},\T_{N,k},\y_{N,k}\}$,
$M\subset c(\W_{N,k})$ and $\mathcal{F}\subset\{f:\ f:c(\W_{N,k})\setminus
M\to\C\}$ be such that: \begin{enumerate}[\rm (\text{T}1)]
  \item\label{A1} $M$ is pluripolar,\footnotemark\footnotetext{Actually we can
  assume less: $M$ is such that for all $j\in\{1,\ldots,N\}$ the set $\{a_j\in
  A_j:\ M_{(\cdot,a_j,\cdot)}\text{ is not pluripolar}\}$ is pluripolar. }
  \item\label{A2} for any $\alpha\in\J$ and any
  $a\in\Az\setminus\Sigma_\a$ the fiber $M_{a,\a}$ is pluripolar,
  \item\label{A3} for any $\alpha\in\J$ and any
  $a\in\Az\setminus\Sigma_\a$
  there exists a closed
  pluripolar set $\t{M}_{a,\a}\subset\Dj$ such that
  $\t{M}_{a,\a}\cap\Aj\subset
  M_{a,\a}$,\footnotemark\footnotetext{When $k=N$ we assume that there exists
  an $\t{M}\subset D_1\times\ldots\times D_N$ closed pluripolar such that
  $\t{M}\cap c(\W_{N,k})\subset M$.}
  \item\label{A4} for any $a\in c(\W_{N,k})\setminus M$ there exists an $r>0$
  such that for all $f\in\mathcal{F}$ there exists an $f_a\in\O(\P(a,r))$
  with $f_a=f$ on $\P(a,r)\cap(c(\W_{N,k})\setminus M)$,
  \item\label{A5} for any $f\in\mathcal{F}$, any $\alpha\in\J$, and any
  $a\in\Az\setminus\Sigma_\a$ there exists
  an
  $\t{f}_{a,\a}\in\O(\Dj\setminus\t{M}_{a,\a})$ such that
  $\t{f}_{a,\a}=f_{a,\a}$ on $\Aj\setminus
  M_{a,\a}$\footnotemark\footnotetext{When $k=N$ we assume that there exists
  an $\t{f}\in\O(D_1\times\ldots\times D_N\setminus\t{M})$ such that $\t{f}=f$
  on $c(\W_{N,k})\setminus M$.}.
\end{enumerate}
Then there exists a relatively closed pluripolar set
$\w{M}\subset\w{\x}_{N,k}$ such that:
\begin{itemize}
  \item $\w{M}\cap c(\W_{N,k})\subset M$,
  \item for any $f\in\mathcal{F}$ there exists
  an $\w{f}\in\O(\w{\x}_{N,k}\setminus\w{M})$ such that $\w{f}=f$ on
  $c(\W_{N,k})\setminus M$,
  \item $\w{M}$ is singular with respect to $\{\w{f}:\
  f\in\mathcal{F}\}$,
  \item if for all $\alpha\in\J$ and all
  $a\in\Az\setminus\Sigma_\a$
  we have $\t{M}_{a,\a}=\varnothing$, then $\w{M}=\varnothing$,
  \item if for all $\alpha\in\J$ and all
  $a\in\Az\setminus\Sigma_\a$ the set $\t{M}_{a,\a}$ is
  thin in $\Dj$, then $\w{M}$ is analytic in
  $\w{\x}_{N,k}$.
\end{itemize}
\end{tw}

Theorem \ref{main} has one immediate and useful consequence, which might be
called main extension theorem for generalized $(N,k)$-crosses with pluripolar
singularities (see analogical theorem for $N$-fold crosses, i.e. Theorem 10.2.6
from \cite{JarPfl 2011}).

\begin{prop}\label{col}
Let $D_j$, $A_j$ and $\Sigma_\a$ be as in Theorem
\ref{main}. Let $$\W_{N,k}\in\{\x_{N,k},\T_{N,k},\y_{N,k}\}.$$ Let
$M\subset\W_{N,k}$ and $\mathcal{F}\subset\o(\W_{N,k}\setminus M)$ be such
that: \begin{enumerate}[\rm (\text{P}1)]
  \item\label{P1} $M\cap c(\W_{N,k})$ is pluripolar,
  \item\label{P2} for any $\alpha\in\J$ and any
  $a\in\Az\setminus\Sigma_\a$ the fiber $M_{a,\a}$ is
  pluripolar and relatively closed in $\Dj$,
  \item\label{P3} for any $a\in c(\W_{N,k})\setminus M$ there exists an $r>0$
  such that for all $f\in\mathcal{F}$ there exists an
  $f_a\in\O(\P(a,r))$ with $f_a=f$ on $\P(a,r)\cap(c(\W_{N,k})\setminus M)$.
\end{enumerate}
Then there exists a relatively closed pluripolar set
$\w{M}\subset\w{\x}_{N,k}$ such that:
\begin{itemize}
  \item $\w{M}\cap c(\W_{N,k})\subset M$,
  \item for any $f\in\mathcal{F}$ there exists
  an $\w{f}\in\O(\w{\x}_{N,k}\setminus\w{M})$ such that $\w{f}=f$ on
  $c(\W_{N,k})\setminus M$,
  \item $\w{M}$ is singular with respect to $\{\w{f}:\
  f\in\mathcal{F}\}$,
  \item if for all $\alpha\in\J$ and all
  $a\in\Az\setminus\Sigma_\a$ we have $M_{a,\a}=\varnothing$, then
  $\w{M}=\varnothing$,
  \item if for all $\alpha\in\J$ and all
  $a\in\Az\setminus\Sigma_\a$ the fiber $M_{a,\a}$ is
  thin in $\Dj$, then $\w{M}$ is analytic in
  $\w{\x}_{N,k}$.
\end{itemize}
\end{prop}

\begin{proof}
Define a set $M':=M\cap c(\W_{N,k})$ and a family
$\mathcal{F}:=\{f|_{c(\W_{N,k})\setminus M}:\ f\in\mathcal{F}\}$. We show
that they satisfy the assumptions of Theorem \ref{main}.

Indeed, for
an $\a\in\J$ and for an $a\in\Az\setminus\Sigma_a$ define
$\t{M}_{a,\a}:=M_{a,\a}$ and $\t{f}_{a,\a}:=f_{a,\a}$. Then:
\begin{itemize}
  \item $M'$ is pluripolar and for all $\a\in\J$ and
  all $a\in\Az\setminus\Sigma_a$ the fibers $M'_{a,\a}$ are
  pluripolar from (P\ref{P1}),
  \item for all $\a\in\J$ and
  all $a\in\Az\setminus\Sigma_a$, the set
  $\t{M}_{a,\a}$ is relatively closed and pluripolar,
  \item for all $f\in\mathcal{F}$, for all $\a\in\J$,
  and all $a\in\Az\setminus\Sigma_a$, the function
  $\t{f}_{a,\a}$ is holomorphic on
  $\Dj\setminus\t{M}_{a,\a}$ (cf. (P\ref{P2}), Definitions
  \ref{holT}, \ref{holY} and Remark \ref{hol_rem}),
  \item from (P\ref{P3}) for any $a\in c(\W_{N,k})\setminus M$ there exists
  an $r>0$ such that for all $f\in\mathcal{F}$ there exists
  an $f_a\in\O(\P(a,r))$ with $f_a=f$ on $\P(a,r)\cap(c(\W_{N,k})\setminus M)$.
\end{itemize}
Thus from Theorem \ref{main} we get the conclusion.
\end{proof}

As we have already mentioned in Section 2, Theorem \ref{main} or, to be more
precise, Proposition \ref{col} implies Main Theorem. The idea of proof is based
on Lemmas 10.2.5, 10.2.7, and 10.2.8 from \cite{JarPfl 2011}.

\begin{proof}[Proof that Proposition \ref{col}
implies Main Theorem]
Let $D_j$, $A_j$,
$\Sigma_{\alpha}$, $\x_{N,k}$, $\T_{N,k}$, $M$, and $\mathcal{F}$ be as in
Theorem \ref{result}. We have to check the assumptions of Proposition \ref{col}. Because
$M$ was pluripolar, for all $\a\in\J$ there exists a
pluripolar set $\Sigma_\a$ such that for all
$a\in\Az\setminus\Sigma_\a$ the fiber $M_{a,\a}$ is
pluripolar. Moreover, because $M$ was relatively closed, all the fibers $M_{a,\a}$ are
relatively closed. To check the last assumption we need the following lemma.

\begin{lem}Under assumptions of Theorem \ref{result} for all $a\in
c(\T_{N,k})\setminus M$ there exists an $r>0$ such that for any
$f\in\mathcal{F}$ there exists an $f_a\in\O(\P(a,r))$ with $f_a=f$ on
$\P(a,r)\cap(c(\T_{N,k})\setminus M)$.
\end{lem}

\begin{proof}
Fix an $a\in c(\T_{N,k})\setminus M$. Let $\rho>0$ be such that $\P (a,\rho)\cap
M=\varnothing$\footnotemark\footnotetext{Recall that $M$ is relatively
closed.}\!. Define new crosses
\begin{align*}
&\x_{N,k}^{a,\rho}:=\X_{N,k}((A_j\cap\P(a_j,\rho),\P(a_j,\rho))_{j=1}^N)\footnotemark\!,\\
&\T_{N,k}^{a,\rho}:=\mathbb{T}_{N,k}((A_j\cap\P(a_j,\rho),\P(a_j,\rho))_{j=1}^N,
(\Sigma_\a\cap\P(a_\a,\rho))_{\a\in\I}).
\end{align*}\footnotetext{From the definition of polydisc in a Riemann domain we
obviously have $\P(a_j,\rho)\subset D_j$, $j=1,\ldots,N$.} Fix an $\a\in\I$ and
an
$a\in(\prod\limits_{j:\a_j=0}(A_j\cap\P(a_j,\rho)))\setminus(\Sigma_\a\cap\P(a_\a,\rho))$. Then $$(\prod\limits_{j:\a_j=1}\P(a_j,\rho))\setminus
M_{a,\a}=\prod\limits_{j:\a_j=1}\P(a_j,\rho),$$so for any $f\in\mathcal{F}$ the
function $\prod\limits_{j:\a_j=1}\P(a_j,\rho)\ni z\mapsto f_{a,\a}(z)$ is holomorphic
and $f\in\o(\T_{N,k}^{a,\rho})$. For $\mathcal{F}=\o^c(\T_{N,k}\setminus M)$
we additionally fix a $b\in\prod\limits_{j:\a_j=1}\P(a_j,\rho)$. We
have$$(\prod\limits_{j:\a_j=1}\P(a_j,\rho))\setminus((\Sigma_\a\cap\P(a_\a,\rho))\cup
M_{b,\a})=(\prod\limits_{j:\a_j=1}\P(a_j,\rho))\setminus(\Sigma_\a\cap\P(a_\a,\rho))$$
and for any $f\in\o^c(\T_{N,k}\setminus M)$ the function
$(\prod\limits_{j:\a_j=1}\P(a_j,\rho))\setminus(\Sigma_\a\cap\P(a_\a,\rho))\ni z\mapsto
f_{b,\a}(z)$ is continuous. Thus $\o^c(\T_{N,k}\setminus
M)\subset\o^c(\T_{N,k}^{a,\rho})$. Using Theorem \ref{cross} for
$\mathcal{F}=\o(\x_{N,k}\setminus M)$ and Theorem \ref{arek} for
$\mathcal{F}=\o^c(\T_{N,k}\setminus M)$, we get $$\forall\
f\in\mathcal{F}\ \exists\ \w{f}_a\in\O(\w{\x}_{N,k}^{a,\rho}):\ \w{f}_a=f\text{
on }\T_{N,k}^{a,\rho}\footnotemark\!.$$\footnotetext{Recall that if for
all $\alpha\in\I$ we have $\Sigma_\a=\varnothing$, then
$\T_{N,k}=\x_{N,k}$ and $\T_{N,k}^{a,\rho}=\x_{N,k}^{a,\rho}$.} Choosing
$r\in(0,\rho)$ small enough to have $\P(a,r)\subset\w{\x}_{N,k}^{a,\rho}$
finishes the proof.
\end{proof}

Now, it is clear that all necessary assumptions are satisfied and we can apply
Proposition \ref{col}. We obtain a pluripolar relatively closed set $\w{M}$
such that for all $f\in\mathcal{F}$ there exists an $\w{f}$ with $\w{f}=f$ on
$c(\T_{N,k})\setminus M$ and $\w{M}$ is singular with respect to $\{\w{f}:\
f\in\mathcal{F}\}$.

Fix an $\a\in\I$ and define $D_\a:=\Dz$ and
$G_\a:=\Dj$. Then both $D_\a$ and $G_\a$ are Riemann domains and
$\w{\x}_{N,k}\subset\overset{\resizebox{4em}{1mm}{\merge}}{D_\a\times G_\a}$ is
a Riemann domain of holomorphy. From Proposition 9.1.4 from \cite{JarPfl 2011} there exists a pluripolar set
$\Sigma'_\a\subset \Az$ such that
$\Sigma_\a\subset\Sigma'_\a$ and for all
$a\in\Az\setminus\Sigma'_\a$ the fiber $\w{M}_{a,\a}$ is
singular with respect to the family $\{\w{f}_{a,\a}:\ f\in\mathcal{F}\}$. In
particular, because every $\w{f}_{a,\a}$ is holomorphic on
$(\w{\x}_{N,k})_{a,\a}\setminus \w{M}_{a,\a}$, we have $\w{M}_{a,\a}\subset M_{a,\a}$ for
$a\in\Az\setminus\Sigma'_\a$.
\indent Hence
$$\w{M}\cap\T'_{N,k}
=\bigcup_{\a\in\I}\{z\in\w{M}\cap\mathcal{X}_\a:\
z_\a\not\in\Sigma'_\a\}\subset M.$$
Now for every $\a\in\I$ and every
$a\in\Az\setminus\Sigma'_\a$ the functions $\w{f}_{a,\a}$ and
$f_{a,\a}$ are holomorphic on the domain $\Dj\setminus
M_{a,\a}$ (thanks to inclusion $\w{M}\cap\T'_{N,k}\subset M$) and equal on
$\Aj\setminus M_{a,\a}$, which is not pluripolar. Thus we have
an equality $\w{f}_{a,\a}=f_{a,\a}$ everywhere on $\Dj\setminus
M_{a,\a}$. Because this equality holds for every $\a$ and $a$, we finally get
$\w{f}=f$ on $\T'_{N,k}\setminus M$.
\end{proof}

\section{Proof of Theorem \ref{main}}\label{proof}

First we show that it is sufficient to prove Theorem \ref{main} with
$\W_{N,k}=\x_{N,k}$.

\begin{lem}
Theorem \ref{main} with $\W_{N,k}=\x_{N,k}$ implies Theorem \ref{main} with
$$\W_{N,k}\in\{\T_{N,k},\y_{N,k}\}.$$
\end{lem}

\begin{proof}
Let $D_j$,
$A_j$, $\Sigma_\alpha$, $\x_{N,k}$, $\T_{N,k}$, $\y_{N,k}$, $M\subset
c(\W_{N,k})$ and $\mathcal{F}\subset\{f:\ f:\W_{N,k}\setminus M\to\c\}$, where
$\W_{N,k}\in\{\T_{N,k}.\y_{N,k}\}$, be like in Theorem \ref{main}. Assume that
this theorem is true with $\W_{N,k}=\x_{N,k}$.

Observe that $c(\y_{N,k})=c(\x_{N,k})\setminus\Delta$ and
$c(\T_{N,k})=c(\x_{N,k})\setminus\t{\Delta}$, where
\begin{align*}
&\Delta:=\bigcap_{\alpha\in\I}\{a\in
A_1\times\ldots\times A_N:\ a_{\alpha}\in\Sigma_\alpha\},\\
&\t{\Delta}:=\bigcap_{\alpha\in\J}\{a\in
A_1\times\ldots\times A_N:\ a_{\alpha}\in\Sigma_\alpha\},
\end{align*} are pluripolar
subsets of $c(\x_{N,k})$, where $a_{\a}$ denotes the projection of $a$ on
$\Az$.

Define $M':=M\cup\t{\Delta}\subset c(\x_{N,k})$. Then $c(\x_{N,k})\setminus\t{\Delta})\setminus
M = c(\x_{N,k})\setminus M'$ and
\begin{align*}
\tag{$\ast$}&c(\T_{N,k})\setminus
M=(c(\x_{N,k})\setminus\Delta)\setminus
M\subset(c(\x_{N,k}\setminus\t{\Delta})\setminus M\textnormal{ for }M\subset
c(\T_{N,k}),\\
\tag{$\ast\ast$}&c(\y_{N,k})\setminus
M=(c(\x_{N,k})\setminus\t{\Delta})\setminus M\textnormal{ for }M\subset
c(\y_{N,k}).
\end{align*}
Define
$\mathcal{F}':=\{f|_{c(\x_{N,k})\setminus M'}:\ f\in\mathcal{F}\}$. Then $M'$ is
pluripolar and for all $\a\in\J$ and all $a\in
\Az\setminus\Sigma_\a$ we have
$\t{\Delta}_{a,\a}=\varnothing$, so $M'_{a,\a}=M_{a,\a}$. Thus $M'$ and the
family $\mathcal{F}'$ satisfies assumptions of Theorem \ref{main} with
$\W_{N,k}=\x_{N,k}$. Then there exists $\w{M'}\subset\w{\x}_{N,k}$, relatively
closed, pluripolar, having all properties from the thesis. Properties
($\ast$) and ($\ast\ast$) give us the conclusion for
$\W_{N,k}\in\{\T_{N,k},\y_{N,k}\}$. \end{proof}\medskip

\begin{proof}[Proof of Theorem \ref{main} with $\W_{N,k}=\x_{N,k}$]
$\ $\\
\indent {\bf Step 1.} Theorem \ref{main} is true for any $N$ when $k=1$ (Theorem
\ref{classic}) and when $k=N$ (in this case we assumed the thesis).

\indent {\bf Step 2.} In particular, theorem is true for $N=2$, $k=1,2$. Assume
we already have Theorem \ref{main} for $(N-1,k)$, where $k\in\{1,\ldots, N-1\}$,
and for $(N,1)$,\ldots ,$(N,k-1)$, where $k\in\{2,\ldots,N-1\}$. We need to prove it for
$(N,k)$.

\indent {\bf Step 3.} 
Fix an $s\in\{1,\ldots ,N\}$ (to simplify the notation let $s=N$).
Let $$Q_N:=\{a_N\in A_N:
M_{(\cdot,a_N)}\text{ is not pluripolar}\}.$$Then $Q_N$ is pluripolar. Define
$$\x_{N-1,k}^{(s)}:=\X_{N-1,k}((A_j,D_j)_{j=1,j\neq s}^{N}),\ s=1,\ldots,N,$$
in particular$$\x_{N-1,k}^{(N)}=\x_{N-1,k}:=\X_{N-1,k}((A_j,D_j)_{j=1}^{N-1}).$$

Fix an $a_N\in A_N\setminus Q_N$ and define a family
$\{f(\cdot,a_N):\ f\in\mathcal{F}\}\subset\{f:\ f:c(\x_{N-1,k})\to\C\}$. Then:
\begin{itemize}
  \item $M_{(\cdot, a_N)}\subset c(\x_{N-1,k})$ is pluripolar.
  \item For any $\alpha'\in\mathcal{Y}^{N-1}_{k}$
  and any
  $a'\in A_0^{\a'}\setminus\Sigma_{\a'}$\footnotemark\footnotetext{By
  $A_0^{\a'}$ we denote the product
  $\prod\limits_{j\in\{1,\ldots,N-1\}:\,\a'_j=0}A_j$. Analogously for
  $A_1^{\a'}$ and $D_1^{\a'}$.} the fiber $(M_{(\cdot,a_N)})_{a',\a'}$ equals
  $M_{a,\a}$, where $a=(a',a_N)$ and $\a=(\a',0)$, so it is pluripolar.
  \item For $\alpha'\in\mathcal{Y}^{N-1}_{k},\ 
  a'\in A_0^{\a'}\setminus\Sigma_{\alpha'}$ we define 
  $\t{M}_{a',\a'}:=\t{M}_{a,\a}$, where
  $a=(a',a_N),\ \a=(\a',0)$. Then
  $\t{M}_{a',\a'}\subset\Dj= D_1^{\a'}$
  is closed, pluripolar and
  $\t{M}_{a',\a'}\cap A_1^{\a'}\subset M_{a',\a'}$.
  \item For any $a'\in c(\x_{N-1,k})\setminus M_{(\cdot,a_N)}$ there exists an
  $r>0$ (the same as for $a=(a',a_N)$) such that for any
  $f\in\mathcal{F}$ there exists $f_{a'}\in\O(\P(a',r))$
  such that $f_{a'}=f(\cdot,a_N)$ on $\P(a',r)\cap(c(\x_{N-1,k})\setminus
  M_{(\cdot,a_N)})$.
  \item For $f\in\mathcal{F}$, for any
  $\alpha'\in\mathcal{Y}^{N-1}_{k}$
  and any $a'\in A_0^{\a'}\setminus\Sigma_{\alpha'}$, define
  $
  \t{f}_{a',\a'}:=\t{f}_{a,\a}\in\O(\Dj\setminus\t{M}_{a,\a})
  =\O(D_1^{\a'}\setminus\t{M}_{a',\a'})$, where $a=(a',a_N),\
  \a=(\a',0)$. Then $\t{f}_{a',\a'}=f_{a',\a'}$ on
  $A_1^{\a'}\setminus (M_{(\cdot,a_N)})_{a',\a'}$.
\end{itemize}
From the inductive assumption we get a relatively closed pluripolar set
$\w{M}_{a_N}\subset\w{\x}_{N-1,k}$ such that:
\begin{itemize}
  \item $\w{M}_{a_N}\cap c(\x_{N-1,k})\subset M_{(\cdot,a_N)}$,
  \item for any $f\in\mathcal{F}$
  there exists an $\w{f}_{a_N}\in\O(\w{\x}_{N-1,k}\setminus\w{M}_{a_N})$ such
  that $\w{f}_{a_N}=f(\cdot,a_N)$ on $c(\x_{N-1,k})\setminus M_{(\cdot,a_N)}$,
  \item $\w{M}_{a_N}$ is singular with respect to $\{\w{f}_{a_N}:\
  f\in\mathcal{F}\}$,
  \item if for all $\alpha'\in\mathcal{Y}^{N-1}_{k}$ and all
  $a'\in A_0^{\a'}\setminus\Sigma_{\alpha'}$, we have
  $\t{M}_{a',\a'}=\varnothing$, then $\w{M}_{a_N}=\varnothing$,
  \item if for all $\alpha'\in\mathcal{Y}^{N-1}_{k}$ and all
  $a'\in A_0^{\a'}\setminus\Sigma_{\alpha'}$, the set
  $\t{M}_{a',\a'}$ is thin in $D_1^{\a'}$, then $\w{M}_{a_N}$
  is analytic in $\w{\x}_{N-1,k}$.
\end{itemize}

Define a new cross
$$\Z_N:=\X(c(\x_{N-1,k}),A_N;\w{\x}_{N-1,k},D_N).$$
Observe that $\Z_N$ with original $M$,
$\Sigma_{(0,1)}:=\Sigma_{(0,\ldots,0,1)}$, $\Sigma_{(1,0)}:=Q_N$, and the family
$\mathcal{F}$ satisfies all the assumptions of Theorem \ref{main} with $N=2$, $k=1$. Indeed:
\begin{itemize}
  \item For all $a'\in c(\x_{N-1,k})\setminus\Sigma_{(0,\ldots,0,1)}$ and all
  $a_N\in A_N\setminus Q_N$ the fibers $M_{(a',\cdot)}$, $M_{(\cdot,a_N)}$ are
  pluripolar from (T\ref{A1}), (T\ref{A2}) and definition of $Q_N$.
  \item For all $a'\in c(\x_{N-1,k})\setminus\Sigma_{(0,\ldots,0,1)}$
  from (T\ref{A3}) there exists an $\t{M}_{a'}\subset D_N$ closed pluripolar
  such that $\t{M}_{a'}\cap A_N\subset M_{(a',\cdot)}$. For $a_N\in A_N\setminus
  Q_N$ set $\t{M}_{a_N}:=\w{M}_{a_N}$. Then $\t{M}_{a_N}$ is closed pluripolar
  in $\w{\x}_{N-1,k}$ and $\t{M}_{a_N}\cap c(\x_{N-1,k})\subset
  M_{(\cdot,a_N)}$.
  \item For all $(a',a_N)\in (c(\x_{N-1,k})\times A_N)\setminus M$ from
  (T\ref{A4}) there exists an $r>0$ such that for all $f\in\mathcal{F}$ there
  exists an $f_{(a',a_N)}\in\O(\mathbb{P}((a',a_N),r))$ such that $$
  f_{(a',a_N)}=f\text{ on }\mathbb{P}((a',a_N),r)\cap(c(\x_{N,k}\setminus M)).$$
  \item For all $a'\in c(\x_{N-1,k})\setminus\Sigma_{\alpha=(0,\ldots,0,1)}$
  from (T\ref{A5}) there exists an $f_{a'}\in\O(D_N\setminus \t{M}_{a'})$ such
  that $f_{a'}=f$ on $A_N\setminus M_{(a',\cdot)}$. For an $a_N\in A_N\setminus
  Q_N$ define $f_{a_N}:=\w{f}_{a_N}$. Then $f_{a_N}\in\O(\w{\x}_{N-1,k}\setminus
  \t{M}_{a_N})$ and $f_{a_N}=f$ on $c(\x_{N-1,k})\setminus M_{(\cdot,a_N)}$.
\end{itemize}
Then there exists an $\w{M}_N\subset\w{\Z}_N$ relatively closed pluripolar
such that:
\begin{itemize}
  \item $\w{M}_N\cap c(\x_{N,k})\subset M$,
  \item for any $f\in\mathcal{F}$ there
  exists an $\w{f}_N\in\O(\w{\Z}_N\setminus\w{M}_N)$ such that $\w{f}_N=f$ on
  $c(\x_{N,k})\setminus M$,
  \item $\w{M}_N$ is singular with respect to $\{\w{f}_N:\ f\in\mathcal{F}\}$,
  \item if for all $a'\in c(\x_{N-1,k})\setminus\Sigma_{(0,\ldots,0,1)}$
  we have $\t{M}_{a'}=\varnothing$ and for all $a_N\in A_N\setminus Q_N$
  we have $\t{M}_{a_N}=\varnothing$, then $\w{M}_N=\varnothing$,
  \item if for all $a'\in c(\x_{N-1,k})\setminus\Sigma_{(0,\ldots,0,1)}$
  the set $\t{M}_{a'}$ is thin in $D_N$ and for all $a_N\in A_N\setminus Q_N$
  the set $\t{M}_{a_N}$ is thin in $\w{\x}_{N-1,k}$, then $\w{M}_N$ is analytic
  in $\w{\Z}_N$.
\end{itemize}

We repeat the reasoning above for all $s=1,\ldots,N-1$, obtaining a family of
functions $\{\w{f}_s\}_{s=1}^N$ such that for any $s\in\{1,\ldots,N\}$ we have
$\w{f}_s=f$ on $c(\x_{N,k})\setminus M$. Define a new function
$$F_f(z):=\left\{
\begin{matrix}
\w{f}_1(z)&\text{ for }&z\in\w{\Z}_1\setminus\w{M}_1\\
\vdots& & \\
\w{f}_N(z)&\text{ for }&z\in\w{\Z}_N\setminus\w{M}_N\\
\end{matrix}\right. .$$

Assume for a moment that we have the following lemma.
\begin{lem}\label{sklejka}
Function $F_f$ is well defined on
$\left(\bigcup\limits_{s=1}^N\Z_s\right)\setminus\left(\bigcup\limits_{s=1}^N\w{M}_s\right)$.
\end{lem}\smallskip

{\bf Step 4.} Define a 2-fold cross
$$\Z:=\mathbb{X}(\x_{N-1,k-1},A_N;\w{\x}_{N-1,k},D_N)
\subset\bigcup_{s=1}^N \Z_s,$$ a pluripolar set
$$\t{M}:=\left(\bigcup_{s=1}^N\w{M}_s\right)\cap(\x_{N-1,k-1}\times A_N)$$
and a family
$$\t{\mathcal{F}}:=\{\t{f}:=F_f|_{(\x_{N-1,k-1}\times
A_N)\setminus\t{M}}:\ f\in\mathcal{F}\}.$$
We show that $\Z$, $\t{M}$, and $\t{\mathcal{F}}$ satisfy the
assumptions of Theorem \ref{main} with $N=1$ and $k=1$.
\begin{itemize}
  \item $\t{M}$ is pluripolar in $\x_{N-1,k-1}\times A_N$, so there exist
  pluripolar sets $P\subset\x_{N-1,k-1}$, $Q\subset A_N$ such that for all
  $z'\in\x_{N-1,k-1}\setminus P$, $a_N\in A_N\setminus Q$, the fibers
  $\t{M}_{(z',\cdot)}$, $\t{M}_{(\cdot,a_N)}$ are pluripolar.
  \item Let $z'\in\x_{N-1,k-1}\setminus P$. Then there exists an $s\in\{1,\ldots
  ,N-1\}$ such that
  \begin{align*}
  \tag{$\star$}\{z'\}\times
  D_N\subset\x_{N-1,k}^{(s)}\times A_s.
  \end{align*}
  Indeed, let $z'\in\x_{N-1,k-1}$. Then
  $z'=z'_\alpha$ for some $\alpha\in\{0,1\}^{N-1}$, $|\alpha|=k-1$,
  where $z'_\a=(z_{\a_1},\ldots,z_{\a_{N-1}})$ and $z_{\a_j}=a_j\in A_j$ when
  $\a_j=0$, $z_{\a_j}=z_j\in D_j$ otherwise. We may assume that
  $z'=(z_1,\ldots,z_{k-1},a_k,\ldots,a_{N-1})$. Set $s=k$. Fix a
  $z_N\in D_N$. Then $(z_1,\ldots,z_{k-1},a_k,\ldots,a_{N-1},z_N)\in\{z'\}\times
  D_N$ and
  $$(z_1,\ldots,z_{k-1},a_k,\ldots,a_{N-1},z_N)\in\x_{N-1,k}^{(s)}\times A_s.$$
  
  Define ${\t{M}}_{z'}:=(\w{M}_{s})_{(z',\cdot)}$. Then
  $\t{M}_{z'}$ is pluripolar relatively closed in $D_N$ and $\t{M}_{z'}\cap A_N\subset \t{M}_{(z',\cdot)}$.\\
  For an $a_N\in A_N\setminus Q$ define $\t{M}_{a_N}:=(\w{M}_{N})_{(\cdot,a_N)}$
  - relatively closed pluripolar in $\w{\x}_{N-1,k}$ such that
  $\t{M}_{a_N}\cap\x_{N-1,k-1}\subset\t{M}_{(\cdot,a_N)}$.
  \item For any $(z',a_N)\in (\x_{N-1,k-1}\times A_N)\setminus \t{M}$ there exist
  an $s\in\{1,\ldots,N-1\}$ and an $r>0$ such that $\P((z',a_N),r)\subset
  \w{\Z}_s\setminus\w{M}_s$. Then $\w{f}_s\in\O(P((z',a_N),r))$ and
  $\w{f}_s=F_f=\t{f}$ on $P((z',a_N),r)\cap((\x_{N-1,k-1}\times A_N)\setminus
  \t{M})$.
  \item For a $z'\in\x_{N-1,k-1}\setminus P$ choose an $s$ to have ($\star$) and
  define $\t{f}_{z'}:=\w{f}_s(z',\cdot)$. Then $\t{f}_{z'}$ is holomorphic on
  $D_N\setminus\t{M}_{z'}$ and equals $\t{f}(z',\cdot)$ on
  $A_N\setminus\t{M}_{(z',\cdot)}$.\\
  For an $a_N\in A_N\setminus Q$ define
  $\t{f}_{a_N}:=\w{f}_s(\cdot,a_N)$. Then $\t{f}_{a_N}$ is holomorphic
  on $\w{\x}_{N-1,k}\setminus\t{M}_{a_N}$ and equals $\t{f}(\cdot,a_N)$ on
  $\x_{N-1,k-1}\setminus\t{M}_{(\cdot,a_N)}$.
\end{itemize}
Now from Theorem \ref{main} there exists a relatively closed
pluripolar set $\w{M}\subset\w{\Z}$ such that:
\begin{itemize}
  \item $\w{M}\cap (\x_{N-1,k-1}\times A_N)\subset \t{M}$, in particular,
  $\w{M}\cap c(\x_{N,k})\subset M$,
  \item for any $f\in\mathcal{F}$ there exists an
  $\w{f}\in\O(\w{\Z}\setminus\w{M})$ such that $\w{f}=\t{f}$ on
  $(\x_{N-1,k-1}\times A_N)\setminus\t{M}$, in particular $\w{f}=f$ on
  $c(\x_{N,k})\setminus M$,
  \item $\w{M}$ is singular with respect to $\{\w{f}:\ f\in\mathcal{F}\}$,
  \item if for all $z'\in\x_{N-1,k-1}\setminus P$ we have
  $\t{M}_{z'}=\varnothing$ and for all $a_N\in A_N\setminus Q$ $\t{M}_{a_N}=\varnothing$, then
  $\w{M}=\varnothing$,
  \item if for all $z'\in\x_{N-1,k-1}\setminus P$ the set $\t{M}_{z'}$ is thin
  in $D_N$ and for all $a_N\in A_N\setminus Q$ the set $\t{M}_{a_N}$ is thin in
  $\w{\x}_{N-1,k}$, then $\w{M}$ is analytic in $\w{\Z}$.
\end{itemize}

Now assume that for any $\alpha\in\J$
and any $a\in\Az\setminus\Sigma_\a$ we have
$\t{M}_{a,\a}=\varnothing$. Then for any $s\in\{1,\ldots,N\}$ and for any
$a_s\in A_s\setminus Q_s$ we have $\w{M}_{a_s}=\varnothing$ what implies that
for all $s\in\{1,\ldots,N\}$ we have $\w{M}_s=\varnothing$. Then from the
definitions of $\t{M}_{z'}$ and $\t{M}_{a_N}$ we get that for any
$z'\in\x_{N-1,k-1}\setminus P$ we have $\t{M}_{z'}=\varnothing$ and for all
$a_N\in A_N\setminus Q$ we have $\t{M}_{a_N}=\varnothing$, thus
$\w{M}=\varnothing$.\\
  
Analogously if for all $\alpha\in\J$ and all
$a\in\Az\setminus\Sigma_\a$ the fiber $\t{M}_{a,\a}$ is
thin in $\Dj$, then for any $s\in\{1,\ldots,N\}$ and any
$a_s\in A_s\setminus Q_s$ the set $\w{M}_{a_s}$ is analytic (thus thin) in
$\w{\x}^{(s)}_{N-1,k}$, so for all $s\in\{1,\ldots,N\}$ the set $\w{M}_s$ is
analytic in $\w{\Z}_s$. Because fibers of analytic sets are also analytic we get that
for any $z'\in\x_{N-1,k-1}\setminus P$ the set $\t{M}_{z'}$ is thin in $D_N$ and
for any $a_N\in A_N\setminus Q$ the set $\t{M}_{a_N}$ is thin in
$\w{\x}_{N-1,k}$. Then, finally, $\w{M}$ is analytic in $\w{\Z}$.\\

Now we show that $\w{\x}_{N,k}\subset\w{\Z}$. First observe that
if $z=(z',z_N)\in\w{\x}_{N,k}$, then $z'\in\w{\x}_{N-1,k}$. From Lemma
\ref{equalh} for $(z_1,\ldots,z_N)=(z',z_N)\in\w{\x}_{N,k}$ we get
\begin{align*}
\tag{$\ddag$}\h_{\x_{N-1,k-1},\w{\x}_{N-1,k}}(z')+\h_{A_N,D_N}(z_N)=\h_{\w{\x}_{N-1,k-1}}(z')+\h_{A_N,D_N}(z_N)\
.\end{align*}
For $z\in\w{\x}_{N-1,k-1}\subset\w{\x}_{N,k}$ ($\ddag$)$=\h_{A_N,D_N}(z_N)$,
which is less than $1$ from properties of relative extremal function, and for
$z\in\w{\x}_{N,k}\setminus\w{\x}_{N-1,k-1}$ we use Lemma \ref{inc}
$$
(\ddag)=\left(\sum_{j=1}^{N-1}\h_{A_j,D_j}(z_j)\right)-k+1+\h_{A_N,D_N}(z_N)<k-1+1=1\
.$$

To show the opposite inclusion take $(z_1,\ldots,z_N)=(z',z_N)\in\w{\Z}$. From
properties of relative extremal function and Lemma \ref{inc} we get
$$
\left(\sum_{j=1}^{N-1}\h_{A_j,D_j}(z_j)\right)+\h_{A_N,D_N}(z_N)\leq\h_{\w{\x}_{N-1,k-1}}(z')+k-1+\h_{A_N,D_N}(z_N)$$
$$\leq
\h_{\x_{N-1,k-1},\w{\x}_{N-1,k}}(z')+\h_{A_N,D_N}(z_N)+k-1<1+k-1=k\ .$$

Thus, it is left to prove Lemma \ref{sklejka}:
\begin{proof}[Proof of Lemma \ref{sklejka}]

Fix $s$ and $p$. We want to show that $\w{f}_s=\w{f}_p$ on
$(\Z_s\cap\Z_p)\setminus (\w{M}_s\cup\w{M}_p)$. To simplify the notation we may
assume that $s=N-1$ and $p=N$.

\indent {\bf Step 1.} Every connected component of $\Z_{N-1}\cap\Z_N$ contains
part of the center.

\indent From the definition of $\Z_{N-1}$ and
$\Z_N$ we have$$\Z_{N-1}\cap\Z_N=(A_1\times\ldots\times A_{N-2}\times
D_{N-1}\times A_N) \cup (A_1\times\ldots A_{N-1}\times
D_N)$$
$$\cup(\w{\x}_{N-2,k}\times A_{N-1}\times A_N).$$ First take
$B_1:=A_1\times\ldots\times A_{N-2}\times A_{N-1}\times D_N$. Since the product
of a not connected set with any set is not connected, connected components of
$B_1$ are products of connected components of $A_j$, $j=1,\ldots,N-1,$ and $D_N$. Since the last set is
connected, every connected component of $B_1$ ''contains'' $D_N$ (in the
sense of last place in the product) thus it contains a part of the center
$A_1\times\ldots\times A_N$.

\indent Case of $B_2:=A_1\times\ldots\times A_{N-2}\times D_{N-1}\times A_N$ is
similar.

\indent Now take $B_3:=\w{\x}_{N-2,k}\times A_{N-1}\times A_N$. As in the
previous cases, since $\w{\x}_{N-2,k}$ is connected, every connected component
of $B_2$ ''contains'' whole $\w{\x}_{N-2,k}$ in the product. Since
$\w{\x}_{N-2,k}$ contains $A_1\times\ldots\times A_{N-2}$, every connected component of $B_2$
must contain part of the center.

\indent {\bf Step 2.} One connected component of $\w{\Z}_{N-1}\cap\w{\Z}_N$
contains whole $\Z_{N-1}\cap\Z_N$.

\indent Intersection $\w{\Z}_{N-1}\cap\w{\Z}_N$ contains cross
$\x_{N,1}$ which is connected and contains the center. Thus the whole center
must lay in one connected component of $\w{\Z}_{N-1}\cap\w{\Z}_N$. Now take any
connented component of $\w{\Z}_{N-1}\cap\w{\Z}_N$ which intersects
$\Z_{N-1}\cap\Z_N$. From Step 1 it must contain part of the center, so there is
only one connected component of $\w{\Z}_{N-1}\cap\w{\Z}_N$ intersecting (thus
containing) $\Z_{N-1}\cap\Z_N$.

\indent {\bf Step 3.} Every connected component of $\w{\Z}_{N-1}\cap\w{\Z}_N$
with $\w{M}_{N-1}\cup\w{M}_N$ deleted is a domain, thus it is a connected component
of $(\w{\Z}_{N-1}\cap\w{\Z}_N)\setminus(\w{M}_{N-1}\cup\w{M}_N)$.

\indent Take any connected component of $\w{\Z}_{N-1}\cap\w{\Z}_N$, name it
$\Omega$. Then $\Omega$ is a domain. The set $\w{M}_{N-1}$ is pluripolar and relatively
closed in $\w{\Z}_{N-1}$, thus it is pluripolar and relatively closed in
$\Omega$, so $\Omega\setminus\w{M}_{N-1}$ is still a domain. Because $\w{M}_N$
is relatively closed and pluripolar in $\w{\Z}_N$, it is relatively closed and
pluripolar in $\Omega\setminus\w{M}_{N-1}$. So
$\Omega\setminus(\w{M}_{N-1}\cup\w{M}_N)$ is a domain.

\indent {\bf Step 4.} One connected component of
$(\w{\Z}_{N-1}\cap\w{\Z}_N)\setminus(\w{M}_{N-1}\cup\w{M}_N)$
contains whole set $(\Z_{N-1}\cap\Z_N)\setminus(\w{M}_{N-1}\cup\w{M}_N)$.

\indent It follows immediately from Step 2 and Step 3.

\indent {\bf Step 5.} $\w{f}_{N-1}=\w{f}_N$ on
$(\Z_{N-1}\cap\Z_N)\setminus(\w{M}_{N-1}\cup\w{M}_N)$.

\indent Let $\Omega$ be a connected component from Step 4. Then both
$\w{f}_{N-1}$ and $\w{f}_N$ are defined on $\Omega$. On the non-pluripolar center we have
$\w{f}_{N-1}=\w{f}_N$. $\Omega$ is a domain and contains the center, so
$\w{f}_{N-1}=\w{f}_N$ on $\Omega$. Moreover, $\Omega$ contains
$(\Z_{N-1}\cap\Z_N)\setminus(\w{M}_{N-1}\cup\w{M}_N)$, what finishes the proof.
\end{proof}

The proof of Theorem \ref{main} is finished.
\end{proof}

\begin{ex}
In the proof of Theorem \ref{main} with $k=1$ we do not need cross $\w{\Z}$ - it
is sufficient to take $\w{\Z}_N$ (see \cite{JarPfl 2010} for details), however
in the case when $k>1$ Step 4 is necessary. Indeed, let $A_1=A_2=A_3=(-1,1)$, $D_1=D_2=D_3=\D$,
$\x_{3,2}:=\X_{3,2}((A_j,D_j)_{j=1}^3)$, $\Z_3:=\X(A_1\times
A_2,A_3;\w{\x}_{2,2},D_3)$. Then $\w{\x}_{2,2}=D_1\times D_2$,
$$\w{\Z}_3:=\{z\in D_1\times D_2\times D_3:\ \h_{A_1\times A_2,D_1\times
D_2}(z_1,z_2)+\h_{A_3,D_3}(z_3)<1\}=$$ $$=\{z\in D_1\times D_2\times D_3:\
\max\{h_{A_j,D_j}(z_j),\ j=1,2\}+\h_{A_3,D_3}(z_3)<1\},$$
and
$\h_{A_j,D_j}(\zeta)=\frac{2}{\pi}\left|\Arg\left(\frac{1+\zeta}{1-\zeta}\right)\right|$,
$\zeta\in\D$, $j=1,2,3$ (see Example 3.2.20 (a) in \cite{JarPfl 2011}).
Take
$z=(0,w,w)$, where $w=\frac{i}{\sqrt{3}}$. Then
$z\in\x_{3,2}$ but $z\not\in\w{\Z}_3$.

\end{ex}

\renewcommand\refname{References}

\end{document}